\documentclass[11pt]{amsart}
\usepackage[latin1]{inputenc}
\usepackage[english]{babel}
\usepackage{amsmath,amssymb,amsthm,amsfonts}

\makeatletter
\@namedef{subjclassname@2010}{ \textup{2010} Mathematics Subject Classification}
\makeatother
\theoremstyle{definition}
\newtheorem{theorem}{Theorem}
\theoremstyle{definition}
\newtheorem{lemma}[theorem]{Lemma}
\theoremstyle{definition}
\newtheorem{corollary}[theorem]{Corollary}
\theoremstyle{definition}
\newtheorem{proposition}[theorem]{Proposition}
\theoremstyle{definition}

\theoremstyle{definition}
\newtheorem{remark}[theorem]{Remark}
\theoremstyle{definition}
\newtheorem{example}[theorem]{Example}
\numberwithin{equation}{section}

\numberwithin{equation}{section}
\DeclareMathOperator{\supp}{supp}

\input{tcilatex}

\begin{document}
\title{Isometric copies of $l^\infty$ in Ces\`aro-Orlicz function spaces}
\author{Tomasz Kiwerski, Pawe\l {} Kolwicz}

\begin{center}
\textbf{{\large {Isometric copies of $l^{\infty }$ in Cesàro-Orlicz\\[0pt]
function spaces}}}
\end{center}

\bigskip

\begin{center}
Tomasz Kiwerski

Faculty of Mathematics, Computer Science and Econometrics\\[0pt]
University of Zielona Góra, \\[0pt]
prof. Z. Szafrana 4a, 65-516 Zielona Góra, Poland\\[0pt]
e-mail address: \verb|tomasz.kiwerski@gmail.com|
\end{center}

\bigskip

\begin{center}
Pawe\l\ Kolwicz

Institute of Mathematics,\\[0pt]
Faculty of Electrical Engineering,\\[0pt]
Pozna\'{n} University of Technology,\\[0pt]
Piotrowo 3A, 60-965 Pozna\'{n}, Poland\\[0pt]
e-mail address: \verb|pawel.kolwicz@put.poznan.pl|
\end{center}

\bigskip \noindent ABSTRACT. We characterize Cesàro-Orlicz function spaces $%
Ces_{\varphi }$ containing order isomorphically isometric copy of $l^{\infty
}$. We discuss also some useful applicable conditions sufficient for the
existence of such a copy.

\bigskip

\begin{flushleft}

\end{flushleft}

\bigskip

\section{INTRODUCTION}

The structure of different types of Cesàro spaces has been widely
investigated during the last decades from the isomorphic as well as
isometric point of view. The spaces generated by the Cesàro operator
(including abstract Cesàro spaces) have been considered by Curbera, Ricker
and Le\'{s}nik, Maligranda in several papers (see \cite{Cur-Ric}, \cite%
{Cur-Ric3}, \cite{Cur-Ric4}, \cite{Cur-Ric6}, \cite{LM15a}, \cite{LM15b}, 
\cite{LM15p}). The classical Cesàro sequence $ces_{p}$ and function $Ces_{p}$
spaces have been studied by many authors (see \cite{AM09}, \cite{AM14} -
also for further references, \cite{AM2008}, \cite{AM2015}, \cite{CJP97}, 
\cite{CMP00}). It has been proved among others that some properties are
fulfilled in the sequence case and are not in function case. Moreover,
sometimes the cases $Ces_{p}[0,1]$ and $Ces_{p}[0,\infty )$ are essentially
different (see an isomorphic description of the Köthe dual of Cesàro spaces
in \cite{AM09} and \cite{LM15a}).

The Cesàro-Orlicz sequence spaces denoted by $ces_{\varphi }$ are
generalization of the Cesàro sequence spaces $ces_{p}$. Of course, the
structure of the spaces $ces_{\varphi }$ is richer than of the space $%
ces_{p} $. The spaces $ces_{\varphi }$ have been studied intensively (see 
\cite{CHPSSz05}, \cite{KK10}, \cite{Ku09} and \cite{MPS}). We want to
investigate the Cesàro-Orlicz function space $Ces_{\varphi }(I)$. The spaces 
$Ces_{\varphi }$ which contain an order isomorphic copy of $l^{\infty }$
(equivalently, are not order continuous) have been characterized in \cite%
{KK16}. The monotonicity properties have been also considered in \cite{KK16}%
. In this paper we want to describe Cesàro-Orlicz function spaces $%
Ces_{\varphi }$ containing order isomorphically isometric copy of $l^{\infty
}$. We discuss also some useful applicable conditions sufficient for the
existence of such a copy$.$ We admit the largest possible class of Orlicz
functions giving the maximal generality of spaces under consideration.

\section{PRELIMINARIES}

Let $\mathbb{R}$, $\mathbb{R_{+}}$ and $\mathbb{N}$ be the sets of real,
nonnegative real and natural numbers, respectively. Denote by $\mu $ the
Lebesgue measure on $I$ and by $L^{0}=L^{0}(I)$ the space of all classes of
real-valued Lebesgue measurable functions defined on $I$, where $I=[0,1]$ or 
$I=[0,\infty )$.

A Banach space $E=(E,\Vert \cdot \Vert )$ is said to be a Banach ideal space
on $I$ if $E$ is a linear subspace of $L^{0}(I)$ and satisfies two
conditions:

\begin{enumerate}
\item if $g\in E$, $f\in L^{0}$ and $|f|\leq |g|$ a.e. on $I$ then $f\in E$
and $\Vert f\Vert \leq \Vert g\Vert $,

\item there is an element $f\in E$ that is positive on whole $I $.
\end{enumerate}

\noindent Sometimes we write $\left\| \cdot\right\|_{E}$ to be sure in which
space the norm has been taken.

For two Banach ideal spaces $E$ and $F$ on $I$ the symbol $E \hookrightarrow
F$ means that the embedding $E\subset F$ is continuous, i.e., there exists
constant a $C>0$ such that $\left\| x\right\|_{F} \le C\left\| x\right\|_{E}$
for all $x\in E$. Moreover, $E=F$ means that the spaces are the same as the
sets and the norms are equivalent.

A Banach ideal space is called order continuous ($E\in \text{(OC)}$ shortly)
if every element of $E$ is order continuous, that is, for each $f\in E$ and
for each sequence $(f_{n})\subset E$ satisfying $0\leq f_{n}\leq |f|$ and $%
f_{n}\rightarrow 0$ a.e. on $I$, we have $\left\Vert f_{n}\right\Vert
\rightarrow 0$. By $E_{a}$ we denote the subspace of all order continuous
elements of $E$. It is worth to notice that in case of Banach ideal spaces
on $I$, $x\in E_{a}$ if and only if $\left\Vert x\chi _{A_{n}}\right\Vert
\rightarrow 0$ for any decreasing sequence of Lebesgue measurable sets $%
A_{n}\subset I$ with empty intersection (see \cite[Proposition 3.5, p. 15]%
{BS88}).

A function $\varphi :[0,\infty )\rightarrow \lbrack 0,\infty ]$ is called an
Orlicz function if:

\begin{enumerate}
\item $\varphi$ is convex,

\item $\varphi (0)=0$,

\item $\varphi $ is neither identically equal to zero nor infinity on $%
(0,\infty )$,

\item $\varphi $ is left continuous on $(0,\infty )$, i.e., $%
\lim_{u\rightarrow b_{\varphi }^{-}}\varphi (u)=\varphi (b_{\varphi })$ if $%
b_{\varphi }<\infty $, where 
\begin{equation*}
b_{\varphi }=\sup \{u>0:\varphi (u)<\infty \}.
\end{equation*}
\end{enumerate}

For more information about Orlicz functions see \cite{Ch96} and \cite{KR61}.

If we denote 
\begin{equation*}
a_{\varphi }=\sup \{u\geq 0:\varphi (u)=0\},
\end{equation*}%
then $0\leq a_{\varphi }\leq b_{\varphi }\leq \infty $. Moreover, $%
a_{\varphi }<\infty $ and $b_{\varphi }>0$, since an Orlicz function is
neither identically equal to zero nor infinity on $(0,\infty )$. The
function $\varphi $ is continuous and non-decreasing on $[0,b_{\varphi })$
and is strictly increasing on $[a_{\varphi },b_{\varphi })$. We use
notations $\varphi >0$, $\varphi <\infty $ when $a_{\varphi }=0$, $%
b_{\varphi }=\infty $, respectively.

We say an Orlicz function $\varphi $ satisfies the condition $\Delta _{2}$
for large arguments ($\varphi \in \Delta _{2}(\infty )$ for short) if there
exists $K>0$ and $u_{0}>0$ such that $\varphi (u_{0})<\infty $ and 
\begin{equation*}
\varphi (2u)\leq K\varphi (u)
\end{equation*}%
for all $u\in \lbrack u_{0},\infty )$. Similarly, we can define the
condition $\Delta _{2}$ for small, with $\varphi (u_{0})>0$ $(\varphi \in
\Delta _{2}(0))$ or for all arguments $(\varphi \in \Delta _{2}(\mathbb{R_{+}%
}))$. These conditions play a crucial role in the theory of Orlicz spaces,
see \cite{Ch96}, \cite{KR61}, \cite{Ma89} and \cite{Mu83}. We will write $%
\varphi \in \Delta _{2}$ in two cases: $\varphi \in \Delta _{2}(\infty )$ if 
$I=[0,1]$ and $\varphi \in \Delta _{2}(\mathbb{R_{+}})$ if $I=[0,\infty )$.

The Orlicz function space $L^{\varphi }=L^{\varphi }(I)$ generated by an
Orlicz function $\varphi $ is defined by 
\begin{equation*}
L^{\varphi }=\{f\in L^{0}(I):I_{\varphi }(f/\lambda )<\infty \ \text{for some%
}\ \lambda =\lambda (f)>0\},
\end{equation*}%
where $I_{\varphi }(f)=\int_{I}\varphi (|f(t)|)dt$ is a convex modular (for
the theory of Orlicz spaces and modular spaces see \cite{Ma89} and \cite%
{Mu83}). The space $L^{\varphi }$ is a Banach ideal space with the
Luxemburg-Nakano norm 
\begin{equation*}
\left\Vert f\right\Vert _{\varphi }=\inf \{\lambda >0:I_{\varphi }(f/\lambda
)\leq 1\}.
\end{equation*}%
It is well known that $\left\Vert f\right\Vert _{\varphi }\leq 1$ if and
only if $I_{\varphi }(f)\leq 1$. Moreover, the set 
\begin{equation}
KL^{\varphi }=KL^{\varphi }(I)=\{f\in L^{0}(I):I_{\varphi }(f)<\infty \},
\label{def: klasa orlicza}
\end{equation}%
will be called the Orlicz class.

The Cesàro operator $C:L^{0}(I)\rightarrow L^{0}(I)$ is defined by 
\begin{equation*}
Cf(x)=\frac{1}{x}\int_{0}^{x}f(t)dt,
\end{equation*}%
for $0<x\in I$. For a Banach ideal space $X$ on $I$ we define an abstract Ces%
àro space $CX=CX(I)$ by 
\begin{equation*}
CX=\{f\in L^{0}(I):C|f|\in X\}
\end{equation*}%
with the norm $\left\Vert f\right\Vert _{CX}=\left\Vert C|f|\right\Vert _{X}$
(see \cite{LM15a}, \cite{LM15b}, \cite{LM15p}).

The Cesàro-Orlicz function space $Ces_{\varphi }=Ces_{\varphi }(I)$ is
defined by $Ces_{\varphi }(I)=CL^{\varphi }(I)$. Consequently, the norm in
the space $Ces_{\varphi }$ is given by the formula 
\begin{equation*}
\left\Vert f\right\Vert _{Ces(\varphi )}=\inf \{\lambda >0:\rho _{\varphi
}(f/\lambda )\leq 1\},
\end{equation*}%
where $\rho _{\varphi }(f)=I_{\varphi }(C|f|)$ is a convex modular. We
always assume that $Ces_{\varphi }\neq \{0\}$. If $I=[0,\infty )$ then $%
Ces_{\varphi }[0,\infty )\neq \{0\}$ if and only if the function $%
x\rightarrow \frac{1}{x}\chi _{\lbrack a,\infty )}(x),x>0$ belongs to $%
L^{\varphi }[0,\infty )$ for some $a>0$ (see Theorem 1 (a) in \cite{LM15a}
and Proposition 3 in \cite{KK16}). However, $Ces_{\varphi }[0,1]\neq \{0\}$
for any Orlicz function $\varphi .$ Indeed, $L^{\varphi }[0,1]$ is symmetric
and $Ces_{\varphi }[0,1]\neq \{0\}$ if and only if $\chi _{\lbrack a,1]}\in
L^{\varphi }[0,1]$ for some $0<a<1$ (see Theorem 1 (b) in \cite{LM15a}).

In this paper we accept the convention that $\sum_{n=m}^{k}x_{n}=0$ if $k<m$.

Note that if $0<a_{\varphi }=b_{\varphi }$ then $L^{\varphi }=L^{\infty }$
and $\left\Vert x\right\Vert _{\varphi }=\frac{1}{b_{\varphi }}\left\Vert
x\right\Vert _{\infty }$, see Example 1 in \cite[\text{p.} 98]{Ma89}.
Consequently, $Ces_{\varphi }=Ces_{\infty }$ in that case (see \cite{AM09}, 
\cite{LM15a}). Therefore we can assume that if $b_{\varphi }<\infty $, then $%
a_{\varphi }<b_{\varphi }$.

\section{Isometric copies of $l^{\infty }$ in $Ces_{\protect\varphi }$}

\noindent Define a set 
\begin{equation*}
C_{\varphi }=C_{\varphi }(I)=\{x\in Ces_{\varphi }(I):\rho _{\varphi
}(kx)<\infty \ \text{for all}\ k>0\}.
\end{equation*}%
Note that $C_{\varphi }=\{0\}$ whenever $b_{\varphi }<\infty .$ If $%
b_{\varphi }=\infty ,$ then $\left( Ces_{\varphi }\right) _{a}=C_{\varphi }$
by Theorem 5 from \cite{KK16} (recall that $\left( Ces_{\varphi }\right)
_{a} $ is the subspace of all order continuous elements in $Ces_{\varphi }$).

We say that a measurable set $\Omega $ is a support of a Banach ideal space $%
E$, we write $\limfunc{supp}E=\Omega $ whenever\newline
$\left( i\right) $ for each $x\in E$ there is a measurable set $A$ with $\mu
\left( A\right) =0$ and $\limfunc{supp}x\subset A\cup \Omega .$\newline
$\left( ii\right) $ there is $x\in E$ such that $\mu \left( \Omega
\backslash \limfunc{supp}x\right) =0.$

\begin{lemma}
\label{supp Ces_fi_a}Suppose $b_{\varphi }<\infty .$ Then $\left( L^{\varphi
}\right) _{a}=\left\{ 0\right\} $ and $\limfunc{supp}\left( Ces_{\varphi
}\right) _{a}=I.$ Moreover, if $0\leq x\in \left( Ces_{\varphi }\right)
_{a}, $ then $\lim_{t\rightarrow 0^{+}}Cx\left( t\right) =0.$
\end{lemma}

\begin{proof}
The equality $\left( L^{\varphi }\right) _{a}=\left\{ 0\right\} $ is well
known, it is \ enough to consider the element $x=\alpha \chi _{A}$ for $%
\alpha >0$ and $0<\mu \left( A\right) <\infty .$ Taking a sequence $\left(
A_{n}\right) $ of measurable subsets of $A$ with $\mu \left( A_{n}\right)
\rightarrow 0$ we conclude that $\left\Vert x\chi _{A_{n}}\right\Vert
_{\varphi }\nrightarrow 0$ because $I_{\varphi }\left( \lambda x\chi
_{A_{n}}\right) =\infty $ for $\lambda >b_{\varphi }/\alpha $ and each $n.$
We will prove that $\limfunc{supp}\left( Ces_{\varphi }\right) _{a}=I.$ Let $%
x=\chi _{\left( a,b\right) }$ for any $0<a<b<m\left( I\right) .$ Take a
decreasing sequence of Lebesgue measurable sets $A_{n}\subset I$ with empty
intersection. Then $\mu \left( A_{n}\cap \left( a,b\right) \right)
\rightarrow 0.$ Set $B_{n}=A_{n}\cap \left( a,b\right) $ and $x_{n}=x\chi
_{B_{n}}$ Then $Cx_{n}\rightarrow 0$ uniformly. Consequently, $\rho
_{\varphi }\left( \lambda x_{n}\right) =I_{\varphi }\left( \lambda
Cx_{n}\right) \rightarrow 0$ for each $\lambda >0,$ which for $I=\left[ 0,1%
\right] $ follows directly and for $I=[0,\infty )$ we need Proposition 3
from \cite{KK16}. \ Consequently,\ $\left\Vert x_{n}\right\Vert
_{Ces(\varphi )}\rightarrow 0.$ Thus $x\in \left( Ces_{\varphi }\right) _{a}$
which gives $\limfunc{supp}\left( Ces_{\varphi }\right) _{a}=I.$\newline
Finally, suppose $\limsup_{t\rightarrow 0^{+}}Cx\left( t\right) =\delta >0.$
Then there is a number $\gamma >0$ and a sequence $\left( t_{k}\right)
_{k=1}^{\infty }$ such that $t_{k}\rightarrow 0^{+}$ and $Cx\left(
t_{k}\right) \geq \delta /2$ for each $k.$ Since the function $Cx$ is
continuous on the interval $\left( 0,m\left( I\right) \right) ,$ for each $k$
there is a measurable set $B_{k}$ of positive measure such that 
\begin{equation*}
Cx\left( t\right) \geq \delta /4
\end{equation*}%
for each $t\in B_{k}.$ Let $A_{n}=\left( 0,\frac{1}{n}\right) $ for $n\in 
\mathbb{N}.$ Consequently, for each $n$ we find a number $k_{n}$ with $\mu
\left( B_{k_{n}}\cap A_{n}\right) >0.$ Thus%
\begin{equation*}
\rho _{\varphi }\left( \lambda x\chi _{A_{n}}\right) =I_{\varphi }\left(
\lambda C\left( x\chi _{A_{n}}\right) \right) \geq I_{\varphi }\left(
\lambda \left( Cx\right) \chi _{A_{n}}\right) \geq I_{\varphi }\left(
\lambda \left( Cx\right) \chi _{B_{k_{n}}\cap A_{n}}\right) =\infty
\end{equation*}
for each $\lambda >4b_{\varphi }/\delta ,$ whence $\left\Vert x\chi
_{A_{n}}\right\Vert _{Ces(\varphi )}\nrightarrow 0.$ Thus $x\notin \left(
Ces_{\varphi }\right) _{a}.$
\end{proof}

\begin{remark}
\label{rem: kopial8}Let $I=[0,1]$ or $I=[0,\infty ).$ \newline
$\left( i\right) $ Suppose there is a sequence $\left( z_{n}\right)
_{n=1}^{\infty }$ in $Ces_{\varphi }\left( I\right) $ of pairwise disjoint
supports satisfying conditions: \newline
$\left( a\right) $ $\left\Vert z_{n}\right\Vert _{Ces(\varphi )}=1$, for
each $n.$\newline
$\left( b\right) $ $\left\Vert \sup_{n}z_{n}\right\Vert _{Ces(\varphi )}=1$.%
\newline
Then $Ces_{\varphi }\left( I\right) $ contains an order isomorphically
isometric copy of $l^{\infty }$.\newline
$\left( ii\right) $ Assume that $\varphi <\infty $ and there is an element $%
z\in $ $Ces_{\varphi }\left( I\right) $ such that $\left\Vert z\right\Vert
_{Ces(\varphi )}=1$ and 
\begin{equation*}
\delta (z):=\inf \Big\{\lambda >0:\rho _{\varphi }\left( \frac{z}{\lambda }%
\right) <\infty \Big\}=1.
\end{equation*}%
Then $Ces_{\varphi }\left( I\right) $ contains an order isomorphically
isometric copy of $l^{\infty }$.\newline
$\left( iii\right) $ Let $b_{\varphi }<\infty .$ Suppose there is an element 
$z\in $ $Ces_{\varphi }\left( I\right) $ such that $\left\Vert z\right\Vert
_{Ces(\varphi )}=1$ and%
\begin{equation}
\rho _{\varphi }\left( \frac{z-u}{\lambda }\right) =\infty  \label{ro-phi}
\end{equation}%
for all $\lambda <1$ and $u\in \left( Ces_{\varphi }\right) _{a}.$ Then $%
Ces_{\varphi }\left( I\right) $ contains an order isomorphically isometric
copy of $l^{\infty }$.
\end{remark}

\begin{proof}
$\left( i\right) $ It is enough to apply Theorem 1 in \cite{Hu98}. Note that
each Banach ideal space (in particular $Ces_{\varphi }$) satisfies the
assumption of Theorem 1 in \cite{Hu98}.\newline
$\left( ii\right) $ We apply Theorem 2 from \cite{Hu98}. Note that if $%
\varphi <\infty $ then $\left( Ces_{\varphi }\right) _{a}=C_{\varphi }$ (by
Theorem 5 from \cite{KK16}). Next, if $\varphi <\infty $ and $Ces_{\varphi
}\neq \left\{ 0\right\} $ then $\limfunc{supp}\left( Ces_{\varphi }\right)
_{a}=\limfunc{supp}Ces_{\varphi }=I$ (for $I=[0,\infty )$ it is enough to
apply Proposition 3 from \cite{KK16})$.$ Moreover, since $L^{\varphi }\in
\left( FP\right) ,$ so $Ces_{\varphi }\in \left( FP\right) $ (see Theorem 1
in \cite{LM15a})$,$ in particular $Ces_{\varphi }$ is monotone complete.
Consequently, the space $Ces_{\varphi }$ satisfies the assumptions of
Theorem 2 from \cite{Hu98}. On the other hand, $Ces_{\varphi }$ is a modular
space generated by a convex modular $\rho _{\varphi }.$ By Theorem 2.1 in 
\cite{GH99} we get 
\begin{equation*}
\text{dist}\left( z,\left( Ces_{\varphi }\right) _{a}\right) =\inf \Big\{%
\lambda >0:\rho _{\varphi }\left( \frac{x}{\lambda }\right) <\infty \Big\}.
\end{equation*}%
Thus the conclusion follows from Theorem 2 from \cite{Hu98}.\newline
$\left( iii\right) $ Note that $\limfunc{supp}\left( Ces_{\varphi }\right)
_{a}=I$ by Lemma \ref{supp Ces_fi_a}. Similarly as in $\left( ii\right) $
above we conclude that the space $Ces_{\varphi }$ satisfies the assumptions
of Theorem 2 from \cite{Hu98}. We prove that $dist\left( z,\left(
Ces_{\varphi }\right) _{a}\right) =1.$ Clearly, $dist\left( z,\left(
Ces_{\varphi }\right) _{a}\right) \leq 1.$ We claim that $\left\Vert
z-u\right\Vert _{Ces(\varphi )}\geq 1$ for each $u\in \left( Ces_{\varphi
}\right) _{a}.$ Let $u\in \left( Ces_{\varphi }\right) _{a}.$ The set $%
\left\{ \lambda >0:\rho _{\varphi }\left( \frac{z-u}{\lambda }\right) \leq
1\right\} $ is nonempty. By the assumption (\ref{ro-phi}) we have%
\begin{equation*}
\left\Vert z-u\right\Vert _{Ces(\varphi )}=\inf \left\{ \lambda >0:\rho
_{\varphi }\left( \frac{z-u}{\lambda }\right) \leq 1\right\} \geq 1,
\end{equation*}%
which proves the claim. Thus $dist\left( z,\left( Ces_{\varphi }\right)
_{a}\right) =1$ and, by Theorem 2 from \cite{Hu98}, $Ces_{\varphi }\left(
I\right) $ contains an order isomorphically isometric copy of $l^{\infty }$.
\end{proof}

The following characterization of order continuity is well known. \bigskip 
\newline
\textbf{Theorem A}.\ (G. Ya. Lozanovski{\u{\i}}, see \cite{Lo69}) A Banach
ideal space $E$ is order continuous if and only if $E$ contains no
isomorphic copy of $l^{\infty }$.

\begin{theorem}
\label{th: kopia01} Let $\varphi $ be an Orlicz function. \newline
$\left( i\right) $ In the case when $\varphi \left( b_{\varphi }\right)
=\infty $ we assume that for each $\epsilon >0$ there exists a constant $%
D=D\left( \varepsilon \right) >0$ such that 
\begin{equation}
\rho _{\varphi }(f)\leq DI_{\varphi }(f)+\epsilon ,  \label{Kl-Kl}
\end{equation}%
for all $f\in KL^{\varphi }[0,1]$ satisfying $\left\vert f\left( t\right)
\right\vert \geq \varphi ^{-1}\left( \varepsilon \right) $ for a.e. $t\in 
\limfunc{supp}\left( f\right) $ (see (\ref{def: klasa orlicza}) for the
definition). If $\varphi \notin \Delta _{2}(\infty )$ then $Ces_{\varphi
}[0,1]$ contains an order isomorphically isometric copy of $l^{\infty }$.%
\newline
$\left( ii\right) $ If $\varphi \in \Delta _{2}(\infty )$ then $Ces_{\varphi
}[0,1]$ does not contain an order isomorphic copy of $l^{\infty }.$
\end{theorem}

\begin{proof}
$(i)$. We divide the proof into two parts.

\noindent (A) Suppose $b_{\varphi }=\infty $. Since $\varphi \notin \Delta
_{2}(\infty )$, for each $i\in \mathbb{N}$ there is a sequence $%
(u_{n}^{i})_{n\in \mathbb{N}}\subset \mathbb{R}_{+}$ such that $%
u_{n}^{i}\nearrow \infty $ as $n\rightarrow \infty $, 
\begin{equation*}
\varphi \left( \left( 1+\frac{1}{i}\right) u_{n}^{i}\right) >2^{n+i}\varphi
(u_{n}^{i})
\end{equation*}%
and $\varphi (u_{n}^{i})\geq 1$ for all $n\in \mathbb{N}$. Moreover, there
exists a sequence of pairwise disjoint intervals $A^{i}$ such that $%
\bigcup_{i=1}^{\infty }A^{i}=[0,1)$ and $\mu (A^{i})=2^{-i}$ for $i\in 
\mathbb{N}$. Consequently, we can choose a sequence of intervals $%
(A_{n}^{i})_{n=1}^{\infty }$ with the following properties:\newline
$\left( a\right) $ $(A_{n}^{i})\subset A^{i}$ for all $n\in \mathbb{N}$,%
\newline
$\left( b\right) $ $A_{n}^{i}\cap A_{m}^{i}=\emptyset $ for all $n\neq m$, $%
n,m\in \mathbb{N}$,\newline
$\left( c\right) $ $\mu (A_{n}^{i})=\frac{2^{-n-i}}{\varphi (u_{n}^{i})}$
for all $n\in \mathbb{N}$,\newline
for all $i\in \mathbb{N}$. This is possible, since $\sum_{n=1}^{\infty }%
\frac{2^{-n-i}}{\varphi (u_{n}^{i})}<2^{-i}$ for all $i\in \mathbb{N}$.
Finally, for $i=1,2,3,\ldots $ define the elements 
\begin{equation*}
x_{i}=\sum_{n=1}^{\infty }u_{n}^{i}\chi _{A_{n}^{i}},
\end{equation*}%
and $x=\sum_{i=1}^{\infty }x_{i}$. From the construction it follows that $%
\supp(x_{i})\cap \supp(x_{j})=\emptyset $ for all $i\neq j$. Moreover, 
\begin{equation*}
I_{\varphi }(x_{i})=\int_{0}^{1}\varphi
(x_{i}(t))dt=\int_{0}^{1}\sum_{n=1}^{\infty }\varphi (u_{n}^{i})\chi
_{A_{n}^{i}}(t)dt
\end{equation*}

\begin{equation*}
=\sum_{n=1}^{\infty }\varphi (u_{n}^{i})\mu (A_{n}^{i})=\sum_{n=1}^{\infty
}\varphi (u_{n}^{i})\frac{2^{-n-i}}{\varphi (u_{n}^{i})}=2^{-i},
\end{equation*}%
and 
\begin{equation*}
I_{\varphi }(x)=\int_{0}^{1}\sum_{i=1}^{\infty }\varphi
(x_{i}(t))dt=\sum_{i=1}^{\infty }\int_{0}^{1}\varphi (x_{i}(t))dt=1.
\end{equation*}%
Without loss of generality we may assume that $x_{i}$ is increasing for each 
$i\in \mathbb{N}$. Consider elements $u_{n}^{i}$ in the following way 
\begin{equation}
\begin{matrix}
u_{1}^{1} & u_{2}^{1} & u_{3}^{1} & \cdots & u_{n}^{1} & \cdots \\ 
u_{1}^{2} & u_{2}^{2} & u_{3}^{2} & \cdots & u_{n}^{2} & \cdots \\ 
u_{1}^{3} & u_{2}^{3} & u_{3}^{3} & \cdots & u_{n}^{3} & \cdots \\ 
\vdots & \vdots & \vdots & \ddots & \vdots & \ddots%
\end{matrix}
\label{array: 1}
\end{equation}%
Order the set $\{u_{1}^{1},u_{1}^{2},u_{2}^{1}\}$ from the smallest to the
largest in the sense of order $\leq $. Denote ordered elements by $%
v_{1},v_{2}$ and $v_{3}$. That is, $v_{1}\leq v_{2}\leq v_{3}$. Since for
each $i\in \mathbb{N}$, $u_{n}^{i}\rightarrow \infty $ as $n\rightarrow
\infty $, so there exist indices $i_{1}^{2},i_{2}^{2},i_{3}^{2}\in \mathbb{N}
$, $i_{1}^{2}>2$ and $i_{2}^{2}>1$ such that $%
u_{i_{1}^{2}}^{1},u_{i_{2}^{2}}^{2},u_{i_{3}^{2}}^{3}\geq v_{3}$. Order the
set $\{u_{i_{1}^{2}}^{1},u_{i_{2}^{2}}^{2},u_{i_{3}^{2}}^{3}\}$ as before
and denote ordered elements by $v_{4},v_{5}$ and $v_{6}$. Clearly, the
sequence $(v_{n})_{n=1}^{6}$ is non-decreasing. Generally, in the $m$-th
step we find indexes $i_{1}^{m},i_{2}^{m},\ldots ,i_{m+1}^{m}\in \mathbb{N}$%
, $i_{1}^{m}\geq i_{1}^{m-1}$, $i_{2}^{m}\geq i_{2}^{m-1},\ldots
,i_{m}^{m}\geq i_{m}^{m-1}$ such that $u_{i_{1}^{m}}^{1},u_{i_{2}^{m}}^{2},%
\ldots ,u_{i_{m+1}^{m}}^{m+1}\geq v_{l_{m}}$ and, ordering the set $%
\{u_{i_{1}^{m}}^{1},u_{i_{2}^{m}}^{2},\ldots ,u_{i_{m+1}^{m}}^{m+1}\}$, we
obtain a sequence $(v_{i})_{i=1}^{l_{m}}$, where $l_{1}=3$ and $%
l_{m}=l_{m-1}+m+1$ for $m\geq 2$, which is non-decreasing. By the induction,
we construct a sequence $(v_{n})_{n=1}^{\infty }$ which is non-decreasing.
Corresponding to each value $v_{n}$ the respective interval $A_{n}^{i}$ we
can denote by $A_{n}$ for simplicity. Define an element 
\begin{equation*}
y=\sum_{n=1}^{\infty }y_{n},
\end{equation*}%
where $y_{n}=v_{n}\chi _{B_{n}}$, $B_{n}=\left( a-\sum_{m=1}^{n}\mu
(A_{m}),a-\sum_{m=1}^{n-1}\mu (A_{m})\right) $ and $a=\sum_{n=1}^{\infty
}\mu (A_{n})$. We have $I_{\varphi }(y)\leq I_{\varphi }(x)\leq 1.$ Applying
condition (\ref{Kl-Kl}) for $\epsilon =1/2$ we find $t\in (0,1]$ such that $%
I_{\varphi }(z)\leq 1/2D$ and $\left\vert z\left( t\right) \right\vert \geq
\varphi ^{-1}\left( 1/2\right) $ for a.e. $t\in \limfunc{supp}\left(
z\right) $, where $z=y\chi _{\lbrack 0,t)}$. Therefore, we have 
\begin{equation*}
\rho _{\varphi }(z)\leq DI_{\varphi }(z)+1/2=1.
\end{equation*}%
Moreover, there exists $n_{0}\in \mathbb{N}$ such that $\sum_{n=n_{0}}^{%
\infty }\mu (B_{n})\leq t$. Set $\lambda >0$. There is $i_{0}$ with $\lambda
\geq 1/i_{0}$. We have 
\begin{equation*}
\rho _{\varphi }((1+\lambda )z)=I_{\varphi }(C((1+\lambda )z))\geq
I_{\varphi }((1+\lambda )z)\geq I_{\varphi }\left( \left( 1+\frac{1}{i_{0}}%
\right) z\right) .
\end{equation*}%
Consider the modular $I_{\varphi }((1+1/i_{0})z)$. By the construction of
element $z$, we have chosen infinitely many terms from the $i_{0}$-th row of
matrix (\ref{array: 1}). Denote them by $u_{n_{k}}^{i_{0}}$ for $%
k=1,2,\ldots $. We have 
\begin{equation*}
\varphi \left( \left( 1+\frac{1}{i_{0}}\right) u_{n_{k}}^{i_{0}}\right)
>2^{n_{k}+i_{0}}\varphi (u_{n_{k}}^{i_{0}}),
\end{equation*}%
for each $k$ and $\mu (A_{n_{k}}^{i_{0}})=2^{-n_{k}-i_{0}}/\varphi
(u_{n_{k}}^{i_{0}})$. Therefore, 
\begin{equation*}
I_{\varphi }\left( \left( 1+\frac{1}{i_{0}}\right) z\right) \geq
\sum_{k=k_{0}}^{\infty }\varphi \left( \left( 1+\frac{1}{i_{0}}\right)
u_{n_{k}}^{i_{0}}\right) \mu (A_{n_{k}}^{i_{0}})\geq \sum_{k=k_{0}}^{\infty
}1=\infty .
\end{equation*}%
Consequently, $\left\Vert z\right\Vert _{Ces(\varphi )}=1$ and, by Remark %
\ref{rem: kopial8}$\left( ii\right) $, we conclude that $Ces_{\varphi }[0,1]$
contains an order isomorphically isometric copy of $l^{\infty }$.

\noindent (B) Suppose $b_{\varphi }<\infty .$ We\ apply Remark \ref{rem:
kopial8}$\left( i\right) $.

\begin{enumerate}
\item[(B1)] Let $\varphi (b_{\varphi })=\infty $. Let $(u_{n})\subset 
\mathbb{R_{+}}$ be the sequence with $u_{n}\nearrow b_{\varphi }^{-}$. Since 
$\varphi (u_{n})\rightarrow \infty $, therefore we can assume that $\varphi
(u_{n})\geq 1$ for all $n\in \mathbb{N}$. Put $a_{n}=1/2^{n}\varphi (u_{n})$
for $n\in \mathbb{N}$ and denote $a=\sum_{n=1}^{\infty }a_{n}$. Define a
sequence of pairwise disjoint open intervals $(A_{n})_{n\in \mathbb{N}%
}\subset \lbrack 0,1]$ as follows 
\begin{equation*}
A_{n}=\left( a-\sum_{k=1}^{n}\frac{1}{2^{k}\varphi (u_{k})}%
,a-\sum_{k=1}^{n-1}\frac{1}{2^{k}\varphi (u_{k})}\right) ,
\end{equation*}%
for all $n\in \mathbb{N}$. Let 
\begin{equation*}
x=\sum_{n=1}^{\infty }u_{n}\chi _{A_{n}}.
\end{equation*}%
We have 
\begin{equation*}
I_{\varphi }(x)=\int_{0}^{\infty }\varphi (x)d\mu =\int_{0}^{\infty
}\sum_{n=1}^{\infty }\varphi (u_{n})\chi _{A_{n}}d\mu
\end{equation*}%
\begin{equation*}
=\sum_{n=1}^{\infty }\int_{A_{n}}\varphi (u_{n})d\mu =\sum_{n=1}^{\infty
}\varphi (u_{n})\mu (A_{n})=\sum_{n=1}^{\infty }\frac{1}{2^{n}}=1.
\end{equation*}%
Now, like in (A), we can find $t\in (0,1]$ with $\rho _{\varphi }(z)\leq 1$,
where $z=x\chi _{\lbrack 0,t)}$. Hence, relabeling if necessary, we may
assume that 
\begin{equation*}
z=\sum_{n=1}^{\infty }u_{n}\chi _{A_{n}}.
\end{equation*}%
Because for all $\lambda >0$ there exists $N=N(\lambda )\in \mathbb{N}$ such
that for all $k\geq N$, $(1+\lambda )u_{k}>b_{\varphi }$, hence 
\begin{equation*}
\rho _{\varphi }((1+\lambda )z)=I_{\varphi }(C((1+\lambda )z))=I_{\varphi
}((1+\lambda )Cz)
\end{equation*}%
\begin{equation*}
\geq I_{\varphi }((1+\lambda )z)=\sum_{n=1}^{\infty }\varphi ((1+\lambda
)u_{n})\mu (A_{n})=\infty .
\end{equation*}%
Therefore $\left\Vert z\right\Vert _{Ces(\varphi )}=1$. We will show that 
\begin{equation*}
\rho _{\varphi }\left( \lambda \left( z-u\right) \right) =\infty
\end{equation*}%
for all $\lambda >1$ and $u\in \left( Ces_{\varphi }\right) _{a}$ which
would imply that $Ces_{\varphi }[0,1]$ contains an order isomorphically
isometric copy of $l^{\infty }$ (see Remark \ref{rem: kopial8}$\left(
iii\right) $). Let $\lambda >1$ and $u\in \left( Ces_{\varphi }\right) _{a}.$
We have 
\begin{equation*}
\rho _{\varphi }\left( \lambda \left( z-u\right) \right) =I_{\varphi }\left(
C\left( \left\vert \lambda \left( z-u\right) \right\vert \right) \right)
\geq I_{\varphi }\left( \lambda C\left\vert \left\vert z\right\vert
-\left\vert u\right\vert \right\vert \right) .
\end{equation*}%
By Lemma \ref{supp Ces_fi_a}, $\lim_{t\rightarrow 0^{+}}C\left\vert
u\right\vert \left( t\right) =0.$ Clearly, $Cz\geq z,$ because $z$ is
nonincreasing. Then, since the function $C\left\vert u\right\vert $ is
continuous, there is $\gamma =\gamma \left( \lambda \right) >0$ such that 
\begin{equation*}
\lambda Cz\left( t\right) \geq \lambda z\left( t\right) >b_{\varphi
}+\lambda C\left\vert u\right\vert \left( t\right)
\end{equation*}%
for $t\in \left( 0,\gamma \right) .$ Consequently,%
\begin{equation*}
I_{\varphi }\left( \lambda C\left( \left\vert \left\vert z\right\vert
-\left\vert u\right\vert \right\vert \right) \right) \geq \int_{0}^{\gamma
}\varphi \left( \lambda \left( Cz\left( t\right) -C\left\vert u\right\vert
\left( t\right) \right) \right) dt=\infty ,
\end{equation*}%
which finishes the proof.

\item[(B2)] Suppose $\varphi (b_{\varphi })<\infty $. Let us define an
element $x=b_{\varphi }\chi _{\lbrack 0,1]}$. Then $x\in Ces_{\varphi }[0,1]$%
. Indeed, 
\begin{equation*}
\rho _{\varphi }(x)=I_{\varphi }(Cx)=I_{\varphi }(x)=\int_{0}^{1}\varphi
(b_{\varphi })d\mu =\varphi (b_{\varphi })<\infty .
\end{equation*}%
There exists $t>0$ with 
\begin{equation*}
\rho _{\varphi }\left( x\chi _{\lbrack 0,t)}\right) =\int_{0}^{t}\varphi
(b_{\varphi })ds+\int_{t}^{1}\varphi \left( t\frac{b_{\varphi }}{s}\right)
ds\leq 1.
\end{equation*}%
Moreover, for each $\lambda >0$ 
\begin{equation*}
\rho _{\varphi }((1+\lambda )x\chi _{\lbrack 0,t)})\geq t\varphi ((1+\lambda
)b_{\varphi })=\infty .
\end{equation*}%
For $z=x\chi _{\lbrack 0,t)}$ we have $\left\Vert z\right\Vert _{Ces(\varphi
)}=1.$ Then, like in (B1) above, we conclude that $Ces_{\varphi }[0,1]$
contains an order isomorphically isometric copy of $l^{\infty }$.
\end{enumerate}

\noindent $(ii)$. If $\varphi \in \Delta _{2}(\infty )$ then $L^{\varphi
}[0,1]\in (\text{OC})$ (see for example \cite[p. 21]{Ma89}). Consequently, $%
Ces_{\varphi }[0,1]\in (\text{OC})$ (see \cite[Lemma 1 (a)]{LM15p}). By
Theorem A, $Ces_{\varphi }[0,1]$ cannot contain isomorphic copy of $%
l^{\infty }$.
\end{proof}

\begin{theorem}
\label{th: kopia08} Let $\varphi $ be an Orlicz function. \newline
$\left( i\right) $ In the case when $\varphi \left( b_{\varphi }\right)
=\infty $ we assume that there exists constant $D>0$ such that 
\begin{equation}
\rho _{\varphi }(f)\leq DI_{\varphi }(f),  \label{Kl-Kl-2}
\end{equation}%
for all $f\in KL^{\varphi }[0,\infty )$ (see (\ref{def: klasa orlicza})). If 
$\varphi \notin \Delta _{2}(\mathbb{R}_{+})$ then the space $Ces_{\varphi
}[0,\infty )$ contains an order isomorphically isometric copy of $l^{\infty
} $.\newline
$\left( ii\right) $ If $\varphi \in \Delta _{2}(\mathbb{R}_{+})$ then $%
Ces_{\varphi }[0,1]$ does not contain an order isomorphic copy of $l^{\infty
}.$
\end{theorem}

\begin{proof}
$(i)$. We divide the proof into three parts.

\noindent (A) Assume that $\varphi >0$ and $\varphi <\infty $. Suppose $%
\varphi \notin \Delta _{2}(\mathbb{R}_{+})$. This means that $\varphi \notin
\Delta _{2}(\infty )$ or $\varphi \notin \Delta _{2}(0)$. In the first case
we can proceed as in the proof of Theorem \ref{th: kopia01}. The proof in
the second case is similar to the proof of Theorem 1 in \cite{KK10}. We
present the details for the reader's convenience.

It is well known that $\varphi \in \Delta _{2}(0)$ if and only if there
exist $L>1$ and $K,u_{0}>0$ such that $\varphi (u_{0})>0$, and $\varphi
(Lu)\leq K\varphi (u)$ for all $u\in \lbrack 0,u_{0}]$ (see \cite[\text{p.} 9%
]{Ch96}). Assume that $\varphi \notin \Delta _{2}(0)$. For each $L>1$ and
every sequence $(K_{n})_{n=1}^{\infty }$ we find a sequence $u_{n}\searrow
0^{+}$ with 
\begin{equation*}
\varphi (Lu_{n})>K_{n}\varphi (u_{n}).
\end{equation*}%
Take a sequence $(K_{n})_{n=1}^{\infty }$ satisfying 
\begin{equation*}
\sum_{n=1}^{\infty }\frac{1}{K_{n}}<\infty .
\end{equation*}%
Let $(\epsilon _{m})_{m=1}^{\infty }$ be any positive, decreasing sequence
converging to zero. For any $m\in \mathbb{N}$ take a sequence $%
(K_{n}^{m})_{n=1}^{\infty }$ of positive reals numbers such that 
\begin{equation*}
\sum_{n=1}^{\infty }\frac{1}{K_{n}^{m}}\leq \frac{1}{2^{m}}.
\end{equation*}%
Now by the first part, for any $m\in \mathbb{N}$ we can find a decreasing
sequence $(u_{n}^{m})_{n=1}^{\infty }$ with $u_{n}^{m}\rightarrow 0$ as $%
n\rightarrow \infty $ and $\varphi ((1+\epsilon
_{m})u_{n}^{m})>K_{n}^{m}\varphi (u_{n}^{m})$ for all $m\in \mathbb{N}$.
Thus 
\begin{equation*}
\sum_{n=1}^{\infty }\frac{\varphi \left( u_{n}^{m}\right) }{\varphi
((1+\epsilon _{m})u_{n}^{m})}<\sum_{n=1}^{\infty }\frac{1}{K_{n}^{m}}\leq 
\frac{1}{2^{m}},
\end{equation*}%
for all $m\in \mathbb{N}$. In view of $u_{n}^{m}\rightarrow 0$ as $%
n\rightarrow \infty $, we can find a subsequence $(n_{k})\subset \mathbb{N}$
such that 
\begin{equation*}
u_{n_{1}}^{1}>u_{n_{2}}^{2}>u_{n_{3}}^{3}>\ldots \enskip.
\end{equation*}%
Hence, without loss of generality, we can assume that for $m>1$, 
\begin{equation*}
u_{m}^{m}<u_{m-1}^{m-1}.
\end{equation*}%
Let 
\begin{equation*}
c_{n}=\frac{1}{\varphi ((1+\epsilon _{n})u_{n}^{n})},
\end{equation*}%
for $n\in \mathbb{N}$. Note that $1\leq c_{n}<\infty $ for all $n\in \mathbb{%
N}$. Define 
\begin{equation*}
f(t)=\sum_{n=1}^{\infty }u_{n}^{n}\chi _{\Omega _{n}}(t),
\end{equation*}%
where $\Omega _{n}=[c_{1}+c_{2}+\cdots +c_{n-1},c_{1}+c_{2}+\cdots
+c_{n-1}+c_{n})\subset \mathbb{R}_{+}$. It is clear that the function $f$ is
decreasing. We have 
\begin{equation*}
I_{\varphi }(f)=\int_{0}^{\infty }\varphi (|f(t)|)dt=\sum_{n=1}^{\infty
}c_{n}\varphi (u_{n}^{n})
\end{equation*}%
\begin{equation*}
\leq \sum_{n=1}^{\infty }\frac{\varphi (u_{n}^{n})}{\varphi ((1+\epsilon
_{n})u_{n}^{n})}\leq \sum_{n=1}^{\infty }\frac{1}{2^{n}}=1.
\end{equation*}%
For any $\epsilon >0$ and for any positive integer $M$, we get 
\begin{equation*}
\int_{M}^{\infty }\varphi ((1+\epsilon )|f(t)|)dt\geq \int_{M_{1}}^{\infty
}\varphi ((1+\epsilon )|f(t)|)dt
\end{equation*}%
\begin{equation*}
=\sum_{n=n_{1}}^{\infty }c_{n}\varphi ((1+\epsilon
)u_{n}^{n})=\sum_{n=n_{1}}^{\infty }\frac{\varphi ((1+\epsilon )u_{n}^{n})}{%
\varphi ((1+\epsilon _{n})u_{n}^{n})}\geq \sum_{n=n_{2}}^{\infty }1=\infty ,
\end{equation*}%
for some $M_{1}\geq M$ and $n_{2}\geq n_{1}$. Since $I_{\varphi }(f)\leq 1$,
there is $N\in (0,\infty )$ such that 
\begin{equation*}
I_{\varphi }(f\chi _{\lbrack N,\infty )})\leq \frac{1}{D}.
\end{equation*}%
Consequently, by the assumption, $\rho _{\varphi }(f\chi _{\lbrack N,\infty
)})\leq 1$. Put $g=f\chi _{\lbrack N,\infty )}$ . Since $f$ is a decreasing
function, for any $\epsilon >0$, 
\begin{equation*}
(1+\epsilon )\frac{1}{t}\int_{N}^{t}|f(x)|dx\geq (1+\epsilon )\left( \frac{%
t-N}{t}\right) |f(t)|
\end{equation*}%
\begin{equation*}
=(1+\epsilon )\left( 1-\frac{N}{t}\right) |f(t)|\geq \left( 1+\frac{\epsilon 
}{2}\right) |f(t)|,
\end{equation*}%
for $t$ large enough. Now, for any $\epsilon >0$ and $N_{1}$ large enough we
get 
\begin{equation*}
\rho _{\varphi }((1+\epsilon )g)=\int_{N}^{\infty }\varphi \left(
(1+\epsilon )\frac{1}{t}\int_{N}^{t}|f(x)|dx\right) dt
\end{equation*}%
\begin{equation*}
\geq \int_{N_{1}}^{\infty }\varphi \left( \left( 1+\frac{\epsilon }{2}%
\right) |f(t)|\right) dt=\infty .
\end{equation*}

We have constructed an element $g\in Ces_{\varphi }[0,\infty )$ such that $%
\rho _{\varphi }(g)\leq 1$ and for every $\epsilon >0$, $\rho _{\varphi
}((1+\epsilon )g)=\infty $. By Remark \ref{rem: kopial8} we conclude that $%
Ces_{\varphi }[0,\infty )$ contains an order isomorphically isometric copy
of $l^{\infty }$.

\noindent (B) Suppose $b_{\varphi }<\infty $ and

\begin{enumerate}
\item[(B1)] $\varphi (b_{\varphi })=\infty $. The proof is similar to the
proof of case (B1) of Theorem \ref{th: kopia01}.

\item[(B2)] $\varphi (b_{\varphi })<\infty $. We just proceed as in the
proof of Theorem \ref{th: kopia01} case (B2). We additionally use the
assumption $Ces_{\varphi }\neq \left\{ 0\right\} $ and Proposition 3 from 
\cite{KK16}.
\end{enumerate}

\noindent (C) Assume that $\varphi <\infty $ and $a_{\varphi }>0$. Put $%
x=a_{\varphi }\chi _{\lbrack 0,\infty )}$. Then $\left\Vert x\right\Vert
_{Ces(\varphi )}=1$ and 
\begin{equation*}
\delta (x)=\inf \{\lambda >0:\rho _{\varphi }(x/\lambda )<\infty \}=1.
\end{equation*}%
Indeed, for any $\lambda >0$ 
\begin{equation*}
\rho _{\varphi }((1+\lambda )x)=I_{\varphi }((1+\lambda )Cx)=I_{\varphi
}((1+\lambda )x)=\varphi ((1+\lambda )a_{\varphi })m([0,\infty ))=\infty .
\end{equation*}%
Therefore, by Remark \ref{rem: kopial8}$\left( ii\right) $, $Ces_{\varphi
}[0,\infty )$ contains an order isomorphically isometric copy of $l^{\infty
} $.

\noindent $(ii)$. We can use analogue arguments as in the proof of Theorem %
\ref{th: kopia01}.
\end{proof}

Now we will discuss the technical assumption from Theorem \ref{th: kopia01}
and Theorem \ref{th: kopia08}. \bigskip

Let us recall the standard form of the integral Hardy inequality (see \cite%
{KMP07}). Suppose $p>1$ and $f$ is a nonnegative $p$-integrable function on $%
I$. Then $f$ is integrable over the interval $(0,x)\cap I$ for each positive 
$x$ and 
\begin{equation*}
\int_{I}\left( \frac{1}{x}\int_{0}^{x}f(t)dt\right) ^{p}dx\leq \left( \frac{p%
}{p-1}\right) ^{p}\int_{I}f(x)^{p}dx.
\end{equation*}

The dilation operator $D_{s}$, $s>0$, defined on $L^{0}(I)$ by 
\begin{equation*}
D_{s}x(t)=x(t/s)\chi _{I}(t/s)=x(t/s)\chi _{\lbrack 0,\min \{1,s\})}(t),
\end{equation*}%
for $t\in I$, is bounded in any symmetric space $E$ on $I$ and $\left\Vert
D_{s}\right\Vert _{E\rightarrow E}\leq \text{max}\{1,s\}$ (see \cite[p. 148]%
{BS88}). For the theory of symmetric (rearrangement invariant) spaces the
reader is referred to \cite{BS88} and \cite{KPS78}. Moreover, the lower and
upper Boyd indices of $E$ are defined by 
\begin{equation*}
p(E)=\lim\limits_{s\rightarrow 0^{+}}\frac{\text{ln}\left\Vert
D_{s}\right\Vert _{E\rightarrow E}}{\text{ln}s},
\end{equation*}%
\begin{equation*}
q(E)=\lim\limits_{s\rightarrow \infty }\frac{\text{ln}\left\Vert
D_{s}\right\Vert _{E\rightarrow E}}{\text{ln}s}.
\end{equation*}%
In particular, they satisfy the inequalities 
\begin{equation*}
1\leq p(E)\leq q(E)\leq \infty .
\end{equation*}%
In the case when $E$ is the Orlicz space $L^{\varphi }$, these indices equal
to the so-called lower and upper Orlicz-Matuszewska indices $\alpha
_{\varphi }$ and $\beta _{\varphi }$ of Orlicz functions generating the
Orlicz spaces, i.e., $\alpha _{\varphi }=p(L^{\varphi })$ and $\beta
_{\varphi }=q(L^{\varphi })$ (see \cite[Proposition 2.b.5 and Remark 2 on
page 140]{LT79}). For more details see \cite{Bo69}, \cite{Bo71}, \cite{Ma85}%
, \cite{Ma89} and \cite{Mu83}. \bigskip \newline
Let us mention the important result about boundedness of the operator $C$.
\bigskip \newline
\textbf{Theorem B.} \cite[\text{p.} 127]{KMP07} For any symmetric space $E$
on $I$ the operator $C:E\rightarrow E$ is bounded if and only if the lower
Boyd index satisfies $p(E)>1$. \bigskip \newline
Note that if $(X,\left\Vert \cdot \right\Vert _{X})$ is a Banach ideal space
then $C:X\rightarrow X$ implies $C$ is bounded. In fact, if $C:X\rightarrow
X $, then $X\hookrightarrow CX$. This means that there is $M>0$ with $%
\left\Vert x\right\Vert _{CX}=\left\Vert C\left\vert x\right\vert
\right\Vert _{X}\leq M\left\Vert x\right\Vert _{X}$ for all $x\in X$, i.e. $%
C $ is bounded. However, it may happen that $X\not\hookrightarrow CX$ (see 
\cite[Proposition 2.1]{DS07}). Moreover, $C:CX\rightarrow X$ is always
bounded (from the definition of $CX$) and $CX$ is so-called optimal domain
of $C$ for $X$ (cf. \cite{DS07} and \cite{LM15p}). The immediate consequence
of Theorem $B$ and the above discussion about indices of the Orlicz space $%
L^{\varphi }$ is a next corollary.

\begin{corollary}
\label{coro: inkluzja} The embedding $L^{\varphi }\hookrightarrow
Ces_{\varphi }$ holds if and only if $\alpha _{\varphi }>1$.
\end{corollary}

\begin{remark}
\label{skonczona-calka}If $f\in KL^{\varphi }[0,1]$ then 
\begin{equation*}
\int_{0}^{t}f(s)ds<\infty
\end{equation*}%
for each $t\in \lbrack 0,1]$, because $KL^{\varphi }[0,1]\hookrightarrow
L^{\varphi }[0,1]\hookrightarrow L^{1}[0,1]$. However, we present also the
direct proof, since we need this fact also in the proof of Proposition \ref%
{prop: gamma08skonczone} in the case of $I=[0,\infty )$. Conversely, suppose 
$\int_{0}^{t_{0}}|f(s)|ds=\infty $ for some $t_{0}\in (0,1]$. Consequently, $%
f$ is unbounded on $(0,t_{0})$. Since $\varphi (u)/u\nearrow a$, $a\leq
\infty $, then there is $a_{0}>0$ such that $\varphi (u)\geq a_{0}u/2$ for $%
u\geq u_{3}$. Setting $I_{0}=\{t\in \lbrack 0,1]:|f(t)|\geq u_{3}\}$ we have 
\begin{equation*}
I_{\varphi }(f)\geq \int_{I_{0}}\varphi (|f(s)|)ds\geq \frac{a_{0}}{2}%
\int_{I_{0}\cap \lbrack 0,t_{0}]}|f(s)|ds=\infty ,
\end{equation*}%
whence $f\notin KL^{\varphi }[0,1]$ and the claim is proved.
\end{remark}

\begin{proposition}
\label{prop: gamma01skonczone} Let $\varphi $ be an Orlicz function with $%
\varphi <\infty $. Consider the following conditions:

\begin{enumerate}
\item There exists $p>1$, a convex function $\gamma $ and constants $%
A,B,u_{0}>0$ such that $\varphi (u_{0})>0$ and for all $u\in \lbrack
u_{0},\infty )$, 
\begin{equation*}
A\gamma (u)\leq \varphi (u)^{1/p}\leq B\gamma (u),
\end{equation*}%
i.e. $\varphi ^{1/p}$ is equivalent to a convex function for "large
arguments".

\item For each $\epsilon >0$ there exists a constant $D=D\left( \varepsilon
\right) >0$ such that 
\begin{equation*}
\rho _{\varphi }(f)\leq DI_{\varphi }(f)+\epsilon ,
\end{equation*}%
for all $f\in KL^{\varphi }[0,1]$ satisfying $\left\vert f\left( t\right)
\right\vert \geq \varphi ^{-1}\left( \varepsilon \right) $ for a.e. $t\in 
\limfunc{supp}\left( f\right) .$ Moreover, the Orlicz class $KL^{\varphi
}[0,1]$ is closed under the operator $C$.

\item $C:L^{\varphi }[0,1]\rightarrow L^{\varphi }[0,1]$.
\end{enumerate}

\bigskip \noindent Then (i)$\Rightarrow $(ii)$\Rightarrow $(iii).
\end{proposition}

\begin{proof}
(i)$\Rightarrow $(ii). Take $\epsilon >0$.\ Let $a_{\varphi }<u_{1}$ be such
that $\varphi (u_{1})=\epsilon .$ \newline
Suppose $u_{1}<u_{0}$ ($u_{0}$ is from the condition $\left( i\right) $). We
claim that we can extend the domain of the function $\gamma $ to the
interval $[u_{1},\infty )$ such that the extended function $\gamma
_{e}:[u_{1},\infty )\rightarrow (0,\infty )$ is still continuous and convex.
Denote by $\gamma _{+}^{\prime }$ the right derivative of $\gamma .$ Since $%
\lim_{u\rightarrow \infty }\varphi \left( u\right) ^{1/p}=\infty $, the
condition $\left( i\right) $ implies that there is $u_{2}\geq u_{0}$ such
that $\gamma _{+}^{\prime }\left( u_{2}\right) \geq 0$ and $\gamma \left(
u_{2}\right) =\gamma \left( u_{0}\right) .$ Taking $\gamma _{e}:=\gamma
(u_{0})\chi _{\lbrack u_{1},u_{2}]}+\gamma \chi _{\left( u_{2},\infty
\right) }$ proves the claim.\newline
Moreover, there exist new constants $A^{\prime },B^{\prime }>0$ such that 
\begin{equation*}
A^{\prime }\gamma _{e}(u)\leq \varphi (u)^{1/p}\leq B^{\prime }\gamma
_{e}(u),
\end{equation*}%
for all $u\geq u_{1}$. In fact, it is sufficient to take $A^{\prime }=\text{%
min}\{A,A^{\prime \prime }\}$ and $B^{^{\prime }}=\text{max}\{B,B^{\prime
\prime }\}$, where 
\begin{equation*}
A^{\prime \prime }=\min \{\varphi (u)^{1/p}/\gamma _{e}(u):u_{1}\leq u\leq
u_{2}\}\text{ and }B^{\prime \prime }=\max \{\varphi (u)^{1/p}/\gamma
_{e}(u):u_{1}\leq u\leq u_{2}\}
\end{equation*}%
(these numbers exist because $\varphi (u)^{1/p}/\gamma _{e}(u)$ is
continuous function on closed interval $[u_{1},u_{2}]$). \newline
If $u_{1}\geq u_{0}$ then we proceed with $\gamma _{e}=\gamma ,$ $A^{\prime
}=A$ and $B^{\prime }=B.$\newline
Let $f\in KL^{\varphi }[0,1]$ be such that $\left\vert f\left( t\right)
\right\vert \geq u_{1}$, for a.e. $t\in \limfunc{supp}\left( f\right) .$ Set 
\begin{equation*}
A_{0}=\{t\in \lbrack 0,1]:C|f|(t)\geq u_{1}\}
\end{equation*}%
and 
\begin{equation*}
B_{0}=[0,1]\backslash A_{0}.
\end{equation*}%
Thus, by Remark \ref{skonczona-calka}, 
\begin{equation*}
C|f|(t)=\frac{1}{t}\int_{0}^{t}|f(s)|ds=\int_{0}^{1}|f(tu)|du<\infty ,
\end{equation*}%
for each $t>0$. Then, by (i),%
\begin{equation*}
\rho _{\varphi }(f)=I_{\varphi }(C|f|)
\end{equation*}%
\begin{equation*}
=\int_{0}^{1}\varphi \left( \int_{0}^{1}|f(tu)|du\right)
dt=\int_{A_{0}}\varphi \left( \int_{0}^{1}|f(tu)|du\right)
dt+\int_{B_{0}}\varphi \left( \int_{0}^{1}|f(tu)|du\right) dt
\end{equation*}%
\begin{equation*}
\leq \left( B^{\prime }\right) ^{p}\int_{A_{0}}\left( \gamma _{e}\left(
\int_{0}^{1}|f(tu)|du\right) \right) ^{p}dt+\varphi (u_{1})
\end{equation*}%
\begin{equation*}
\leq \left( B^{\prime }\right) ^{p}\int_{A_{0}}\left( \gamma _{e}\left(
\int_{0}^{1}|f(tu)|du\right) \right) ^{p}dt+\epsilon .
\end{equation*}%
For each $t\in \lbrack 0,1]$, let 
\begin{equation*}
C_{0}^{t}=\{u\in \lbrack 0,1]:|f(tu)|\geq u_{1}\}\text{ and }%
D_{0}^{t}=[0,1]\backslash C_{0}^{t}.
\end{equation*}%
Since $\left\vert f\left( s\right) \right\vert \geq u_{1}$, for a.e. $s\in 
\limfunc{supp}f,$ so%
\begin{equation}
u_{1}\leq C|f|(t)=\int_{0}^{1}|f(tu)|du=\int_{C_{0}^{t}}|f(tu)|du
\label{C0_t}
\end{equation}%
for $t\in A_{0}.$ Consequently, by Hardy and Jensen inequality, we get 
\begin{equation*}
\rho _{\varphi }(f)=I_{\varphi }(C|f|)\leq \left( B^{\prime }\right)
^{p}\int_{A_{0}}\left( \int_{C_{0}^{t}}\gamma _{e}(|f(tu)|)du\right)
^{p}dt+\epsilon
\end{equation*}%
\begin{equation*}
\leq \left( B^{\prime }\right) ^{p}\int_{A_{0}}\left( \int_{C_{0}^{t}}\left(
A^{\prime }\right) ^{-1}\varphi (|f(tu)|)^{1/p}du\right) ^{p}dt+\epsilon
\end{equation*}%
\begin{equation*}
\leq \left( B^{\prime }\right) ^{p}\int_{0}^{1}\left( \int_{C_{0}^{t}}\left(
A^{\prime }\right) ^{-1}\varphi (|f(tu)|)^{1/p}du\right) ^{p}dt+\epsilon
\end{equation*}%
\begin{equation*}
\leq \left( \frac{B^{\prime }}{A^{\prime }}\right) ^{p}\int_{0}^{1}\left(
\int_{0}^{t}\varphi (|f(s)|)^{1/p}\frac{ds}{t}\right) ^{p}dt+\epsilon
\end{equation*}%
\begin{equation*}
=\left( \frac{B^{\prime }}{A^{\prime }}\right) ^{p}\int_{0}^{1}\left(
C(\varphi (|f|)^{1/p})(t)\right) ^{p}dt+\epsilon \leq \left( \frac{B^{\prime
}p}{A^{\prime }(p-1)}\right) ^{p}\int_{0}^{1}\varphi (|f(t)|)dt+\epsilon
\end{equation*}%
\begin{equation*}
=\left( \frac{B^{\prime }p}{A^{\prime }(p-1)}\right) ^{p}I_{\varphi
}(f)+\epsilon .
\end{equation*}%
The proof is finished with $D=\left( \frac{B^{\prime }p}{A^{\prime }(p-1)}%
\right) ^{p}$ and the constant $D$ depends only on $\epsilon $. \newline
Finally, we will show that the condition $\left( i\right) $ implies that the
Orlicz class $KL^{\varphi }[0,1]$ is closed under the operator $C.$ Take $%
f\in KL^{\varphi }[0,1].$ We apply the above proof with the function $\gamma 
$. Denote%
\begin{equation*}
I_{1}=\left\{ t\in \left[ 0,1\right] :\left\vert f\left( t\right)
\right\vert \geq u_{0}\right\} ,I_{2}=\left\{ t\in \left[ 0,1\right]
:\left\vert f\left( t\right) \right\vert <u_{0}\right\}
\end{equation*}%
and%
\begin{equation*}
\widetilde{f}=\left\vert f\right\vert \chi _{I_{1}}+u_{0}\chi _{I_{2}}.
\end{equation*}%
Clearly, $\left\vert f\right\vert \leq \widetilde{f}.$ Moreover, 
\begin{equation*}
\rho _{\varphi }(f)=I_{\varphi }(C|f|)\leq I_{\varphi }(C\widetilde{f}%
)=\int_{0}^{1}\varphi \left( \int_{0}^{1}\widetilde{f}(tu)du\right) dt.
\end{equation*}%
Note that $\widetilde{f}\geq u_{0}$ whence $\int_{0}^{1}\widetilde{f}%
(tu)du\geq u_{0}$ for $t\in \left[ 0,1\right] .$ Consequently, by Hardy and
Jensen inequality, similarly as in the above proof, applying condition $%
\left( i\right) $ we obtain%
\begin{equation*}
\rho _{\varphi }(f)\leq B^{p}\int_{0}^{1}\left( \gamma \left( \int_{0}^{1}%
\widetilde{f}(tu)du\right) \right) ^{p}dt\leq B^{p}\int_{0}^{1}\left(
\int_{0}^{1}\gamma (\widetilde{f}(tu))du\right) ^{p}dt
\end{equation*}%
\begin{equation*}
\leq \left( \frac{Bp}{A(p-1)}\right) ^{p}I_{\varphi }(\widetilde{f})=\left( 
\frac{Bp}{A(p-1)}\right) ^{p}\int_{0}^{1}\varphi \left( \left\vert
f\right\vert \chi _{I_{1}}+u_{0}\chi _{I_{2}}\right)
\end{equation*}%
\begin{equation*}
=\left( \frac{Bp}{A(p-1)}\right) ^{p}\left( I_{\varphi }(\left\vert
f\right\vert \chi _{I_{1}})+\varphi \left( u_{0}\right) \right) \leq \left( 
\frac{Bp}{A(p-1)}\right) ^{p}\left( I_{\varphi }(f)+\varphi \left(
u_{0}\right) \right) <\infty .
\end{equation*}%
\newline

\noindent (ii)$\Rightarrow $(iii). Take $x\in L^{\varphi }[0,1]$. Of course, 
$I_{\varphi }(\lambda x)<\infty $ for some $\lambda >0$. This means that $%
\lambda x\in KL^{\varphi }[0,1]$. Therefore $\rho _{\varphi }(\lambda
x)<\infty $ from (ii). Consequently, $I_{\varphi }(C|\lambda x|)<\infty $,
i.e. $C(\lambda x)\in L^{\varphi }[0,1]$. Because operator $C$ is
homogeneous we have $\lambda Cx\in L^{\varphi }[0,1]$. Whence $Cx\in
L^{\varphi }[0,1]$.
\end{proof}

\begin{proposition}
\label{prop: gamma08skonczone} Let $\varphi$ be an Orlicz function with $%
\varphi < \infty$. Consider the following conditions:

\begin{enumerate}
\item There exists $p>1$, a convex function $\gamma $ and constants $A,B>0$
such that for all $u\in \lbrack 0,\infty )$ 
\begin{equation*}
A\gamma (u)\leq \varphi (u)^{1/p}\leq B\gamma (u),
\end{equation*}%
i.e. $\varphi ^{1/p}$ is equivalent to a convex function for "all arguments".

\item There exists a constant $D>0$ such that 
\begin{equation*}
\rho _{\varphi }(f)\leq DI_{\varphi }(f),
\end{equation*}%
for all $f\in KL^{\varphi }[0,\infty )$. In particular, the Orlicz class $%
KL^{\varphi }[0,\infty )$ is closed under the operator $C$.

\item $C:L^{\varphi }[0,\infty )\rightarrow L^{\varphi }[0,\infty )$.

\item There exists $x_{0}\in \lbrack 0,\infty )$ such that $%
\int_{x_{0}}^{\infty }\varphi (t^{-1})dt<\infty $.
\end{enumerate}

\bigskip \noindent We have the following implications: (i)$\Rightarrow$(ii)$%
\Rightarrow$(iii)$\Rightarrow$(iv).
\end{proposition}

\begin{proof}
The proof of implication (i)$\Rightarrow $(ii) follows by Jensen inequality
and Hardy inequality and is similar to the proof of the same implication in
Proposition \ref{prop: gamma01skonczone}.

\noindent (ii)$\Rightarrow $(iii). We go analogously as in the respective
proof of Proposition \ref{prop: gamma01skonczone}.

\noindent (iii)$\Rightarrow $(iv). There exists $x_{0}$ such that $\varphi
(x_{0}^{-1})<\infty $. If $f=x_{0}^{-1}\chi _{\lbrack 0,x_{0})}$, then $%
I_{\varphi }(f)<\infty $, i.e. $f\in L^{\varphi }[0,1]$ and 
\begin{equation*}
Cf(t)=\frac{1}{t}\int_{0}^{t}x_{0}^{-1}\chi _{\lbrack
0,x_{0})}(s)ds=x_{0}^{-1}\chi _{\lbrack 0,x_{0})}(t)+\frac{1}{t}\chi
_{\lbrack x_{0},\infty )}(t).
\end{equation*}%
Therefore, by (iii), $Cf\in L^{\varphi }[0,1]$ and 
\begin{equation*}
\rho _{\varphi }(\lambda f)=I_{\varphi }(\lambda Cf)=x_{0}\varphi (\lambda
x_{0}^{-1})+\int_{x_{0}}^{\infty }\varphi \left( \frac{\lambda }{t}\right)
dt<\infty ,
\end{equation*}%
for some $\lambda >0$. After changing variables in the last integral we get 
\begin{equation*}
\lambda \int_{x_{0}/\lambda }^{\infty }\varphi \left( \frac{1}{u}\right)
du<\infty ,
\end{equation*}%
and the proof is finished.
\end{proof}

Note that if $b_{\varphi }<\infty $ and $\varphi (b_{\varphi })=\infty $ we
also apply conditions (\ref{Kl-Kl}) and (\ref{Kl-Kl-2}) to prove the
existence of isometric copy of $l^{\infty }$ in $Ces_{\varphi }.$ Thus we
should discuss similar conditions from Proposition \ref{prop:
gamma01skonczone} and \ref{prop: gamma08skonczone} in that case. Below we
add also the case $b_{\varphi }<\infty $, $\varphi (b_{\varphi })<\infty $.

\begin{proposition}
\label{prop: gamma08skacze} Let $I=[0,\infty )$ and $\varphi $ be an Orlicz
function. Consider the following conditions:

\begin{enumerate}
\item Let $b_{\varphi }<\infty $, $\varphi (b_{\varphi })<\infty $ and there
exists $p>1$, a convex function $\gamma $ and constants $A,B>0$ such that
for all $u\in \lbrack 0,b_{\varphi }]$ 
\begin{equation*}
A\gamma (u)\leq \varphi (u)^{1/p}\leq B\gamma (u).
\end{equation*}

\item Let $b_{\varphi }<\infty $, $\varphi (b_{\varphi })=\infty $ and there
exists $p>1$, a convex function $\gamma $ and constants $A,B>0$ such that
for all $u\in \lbrack 0,b_{\varphi })$ 
\begin{equation*}
A\gamma (u)\leq \varphi (u)^{1/p}\leq B\gamma (u).
\end{equation*}

\item There exists a constant $D>0$ such that 
\begin{equation*}
\rho _{\varphi }(f)\leq DI_{\varphi }(f),
\end{equation*}%
for all $f\in KL^{\varphi }(I)$. In particular, the Orlicz class $%
KL^{\varphi }(I)$ is closed under the operator $C$.

\item $C:L^{\varphi }(I)\rightarrow L^{\varphi }(I)$.

\item there exists $x_{0}\in \lbrack 0,\infty )$ such that $%
\int_{x_{0}}^{\infty }\varphi (t^{-1})dt<\infty $.
\end{enumerate}

\bigskip \noindent We have the following implications: $(\text{i}%
)\Rightarrow (\text{iii})\Rightarrow (\text{iv})\Rightarrow (\text{v})$ and $%
(\text{ii})\Rightarrow (\text{iii})$.
\end{proposition}

\begin{proof}
It is enough to observe that if $f\in KL^{\varphi }$ then $f[\supp%
(f)]\subset \lbrack 0,\varphi (b_{\varphi })]$ if $\varphi (b_{\varphi
})<\infty $ and $f[\supp(f)]\subset \lbrack 0,\varphi (b_{\varphi }))$ if $%
\varphi (b_{\varphi })=\infty $ (up to the set of measure zero). Now we
apply the proofs of Propositions \ref{prop: gamma01skonczone} and \ref{prop:
gamma08skonczone}.
\end{proof}

\begin{proposition}
\label{prop: gamma01skacze} Let $I=[0,1]$ and $\varphi $ be an Orlicz
function. Consider the following conditions:

\begin{enumerate}
\item Let $b_{\varphi }<\infty $, $\varphi (b_{\varphi })<\infty $ and there
exists $p>1$, a convex function $\gamma $ and constants $A,B,\varphi
(u_{0})>0$ such that for all $u\in \lbrack u_{0},b_{\varphi }]$ 
\begin{equation*}
A\gamma (u)\leq \varphi (u)^{1/p}\leq B\gamma (u).
\end{equation*}

\item Let $b_{\varphi }<\infty $, $\varphi (b_{\varphi })=\infty $ and there
exists $p>1$, a convex function $\gamma $ and constants $A,B,\varphi
(u_{0})>0$ such that for all $u\in \lbrack u_{0},b_{\varphi })$ 
\begin{equation*}
A\gamma (u)\leq \varphi (u)^{1/p}\leq B\gamma (u).
\end{equation*}

\item For each $\epsilon \in \left( 0,\varphi (b_{\varphi })\right) $ there
exists a constant $D>0$ such that 
\begin{equation*}
\rho _{\varphi }(f)\leq DI_{\varphi }(f)+\epsilon ,
\end{equation*}%
for all $f\in KL^{\varphi }(I)$ satisfying $\left\vert f\left( t\right)
\right\vert \geq \varphi ^{-1}\left( \epsilon \right) $ for a.e. $t\in 
\limfunc{supp}\left( f\right) .$ Moreover, the Orlicz class $KL^{\varphi
}[0,1]$ is closed under the operator $C.$

\item For each $\epsilon >0$ there exists a constant $D>0$ such that 
\begin{equation*}
\rho _{\varphi }(f)\leq DI_{\varphi }(f)+\epsilon ,
\end{equation*}%
for all $f\in KL^{\varphi }(I)$ satisfying $\left\vert f\left( t\right)
\right\vert \geq \varphi ^{-1}\left( \epsilon \right) $ for a.e. $t\in 
\limfunc{supp}\left( f\right) .$ Moreover, the Orlicz class $KL^{\varphi
}[0,1]$ is closed under the operator $C.$

\item $C:L^{\varphi }(I)\rightarrow L^{\varphi }(I)$.
\end{enumerate}

\bigskip \noindent We have the following implications: $(\text{i}%
)\Rightarrow (\text{iii})\Rightarrow (v)$ and $(\text{ii})\Rightarrow (\text{%
iv})\Rightarrow (\text{v})$.
\end{proposition}

\begin{proof}
The proof is similar to the proof of Proposition \ref{prop: gamma01skonczone}
and \ref{prop: gamma08skacze}.
\end{proof}

\begin{remark}
\label{rem: implikacje}
\end{remark}

\begin{enumerate}
\item Condition (iv) in Proposition \ref{prop: gamma08skonczone} and
condition (v) in Proposition \ref{prop: gamma08skacze} is equivalent to $%
Ces_{\varphi }[0,\infty )\neq \{0\}$. This follows from Theorem 1 (a) in 
\cite{LM15a}, see also Proposition 3 in \cite{KK16}.

\item Recall that the condition (ii) from Propositions \ref{prop:
gamma01skonczone}, \ref{prop: gamma08skonczone} and condition (iii) from
Propositions \ref{prop: gamma08skacze}, \ref{prop: gamma01skacze} have been
applied to prove the existence of order isomorphically isometric copy of $%
l^{\infty }$ in $Ces_{\varphi }[0,1]$ - see Theorems \ref{th: kopia01}, \ref%
{th: kopia08}. \newline
Notice also that condition $C:L^{\varphi }\rightarrow L^{\varphi }$ has been
used in \cite{KK16} to prove the criteria for order continuity of $%
Ces_{\varphi }$ (equivalently for the existence of isomorphic copy of $%
l^{\infty }$ in $Ces_{\varphi }$, see Theorem A).
\end{enumerate}

\bigskip A function $f : \mathbb{R}_{+} \rightarrow \mathbb{R}_{+}$ is said
to be pseudo-increasing for all arguments (for small argument or large
arguments) whenever there exist constant $M > 0$, $u_0 \ge 0$ with $f(u) \le
Mf(v)$ for all $0 \le u < v$ ($0 \le u < v \le u_0$ or $u_0 \le u < v$,
respectively). The following useful characterization is well known.

\begin{lemma}
Assume $\varphi$ is an Orlicz function and $p>1$. Function $%
\varphi(u)^{1/p}/u$ is pseudo-increasing (for small arguments, large
arguments or all arguments) if and only if $\varphi^{1/p}$ is equivalent to
convex function $\gamma$ (for small arguments, large arguments or all
arguments respectively).
\end{lemma}

\begin{proof}
$(\Rightarrow)$. See \cite[Theorem 1.6]{KMP00}.

\noindent $(\Leftarrow)$. We were not able to find that simple proof, so we
present it for the reader's convenience. Note that if $\gamma$ is convex
function then for all $t < s$ 
\begin{equation*}
\gamma (s)=\gamma \left(\frac{s}{t} t \right) \ge \frac{s}{t} \gamma(t),
\end{equation*}
therefore 
\begin{equation*}
\frac{\gamma(t)}{t} \le \frac{\gamma(s)}{s}.
\end{equation*}
Suppose $\varphi^{1/p}$ is equivalent to convex function $\gamma$ for all
arguments, i.e. there exists constants $A,B>0$ such that for all $u\in
[0,\infty)$ 
\begin{equation*}
A\frac{\gamma(u)}{u} \le \frac{\varphi(u)^{1/p}}{u} \le B \frac{\gamma(u)}{u}%
.
\end{equation*}
For $u<v$ we have 
\begin{equation*}
\frac{\varphi(u)^{1/p}}{u} \le B \frac{\gamma(u)}{u} \le B \frac{\gamma(v)}{v%
} \le \frac{B}{A} \frac{\varphi(v)^{1/p}}{v}.
\end{equation*}
\end{proof}

Applying Theorem \ref{th: kopia01} and \ref{th: kopia08}, Proposition \ref%
{prop: gamma01skonczone}, \ref{prop: gamma08skonczone}, \ref{prop:
gamma08skacze} and \ref{prop: gamma01skacze} we get following corollary.

\begin{corollary}
Suppose in the case when $\varphi (b_{\varphi })=\infty $ that there is $p>1$
such that the function $\varphi ^{1/p}$ is equivalent to a convex function
for large arguments if $I=[0,1]$ or for all arguments if $I=[0,\infty )$ (as
we mean in the Proposition \ref{prop: gamma01skonczone}, \ref{prop:
gamma08skonczone}, \ref{prop: gamma08skacze} and \ref{prop: gamma01skacze}).
If $\varphi \notin \Delta _{2}$ then the corresponding Cesàro function space 
$Ces_{\varphi }(I)$ contains an order isomorphically isometric copy of $%
l^{\infty }$.
\end{corollary}

\begin{example}
Let $\varphi (u)=e^{u}-1$ for $u\geq 0$. Then $\varphi \notin \Delta
_{2}(\infty )$ and for all $p>1$ the function $\psi =\varphi ^{1/p}$ is
equivalent to a convex function for large arguments. In fact, 
\begin{equation*}
\limsup\limits_{u\rightarrow \infty }\frac{e^{2u}-1}{e^{u}-1}=\infty ,
\end{equation*}%
and 
\begin{equation*}
\psi ^{\prime \prime }(u)=\frac{1}{p^{2}}e^{u}\left( e^{u}-1\right) ^{\frac{1%
}{p}-2}\left( e^{u}-p\right) >0
\end{equation*}%
for all $p>1$ and $u>\ln p$.
\end{example}

\begin{example}
It may happen that $\varphi\notin\Delta_2(\infty)$ and for any $p>1$
function $\varphi^{1/p}$ is not equivalent to any convex function for large
arguments. Indeed, take an Orlicz function $\varphi$ with $%
\varphi\notin\Delta_2(\infty)$ and $\varphi^{\ast}\notin\Delta_2(\infty)$
(see for example \cite[p. 28]{KR61}). Since $\varphi^{\ast}\notin\Delta_2(%
\infty)$ so $\alpha_\varphi = 1$ (see \cite{Ma85}). Consequently, $C$ is not
bounded. Now we apply Proposition \ref{prop: gamma01skonczone}.
\end{example}

\begin{center}
ACKNOWLEDGEMENTS
\end{center}

The second author (Pawe\l {} Kolwicz) is supported by the Ministry of
Science and Higher Education of Poland, grant number 04/43/DSPB/0089.

\end{document}